\documentclass[12pt]{amsart}

\usepackage{a4,amsmath,amssymb,url}
\usepackage{enumitem}

\usepackage{etoolbox}
\patchcmd{\subsection}{\bfseries}{\itshape}{}{}
\patchcmd{\subsection}{-.5em}{.5em}{}{}

\numberwithin{equation}{section}

\theoremstyle{plain}
\newtheorem{thm}{Theorem}[section]
\newtheorem{lem}[thm]{Lemma}
\newtheorem{cor}[thm]{Corollary}
\newtheorem{obs}[thm]{Observation}
\newtheorem{prop}[thm]{Proposition}
\newtheorem*{claim}{Claim}

\theoremstyle{definition}
\newtheorem{conj}[thm]{Conjecture}
\newtheorem{que}[thm]{Question}

\newtheorem{exa}[thm]{Example}

\theoremstyle{remark}
\newtheorem{rem}[thm]{Remark}

\newcommand{\A}{{\mathbf A}}
\newcommand{\B}{{\mathbf B}}
\newcommand{\C}{{\mathbf C}}
\newcommand{\F}{{\mathbf F}}
\newcommand{\G}{{\mathbf G}}

\renewcommand{\L}{\mathbf{L}}
\newcommand{\M}{{\mathbf{M}}}
\newcommand{\Nl}{{\mathbf{N}}}

\renewcommand{\S}{{\mathbf{S}}}

\newcommand{\T}{\mathbf{T}}

\newcommand{\N}{{\mathbb{N}}}
\newcommand{\Z}{{\mathbb{Z}}}

\newcommand{\fF}{{\mathcal F}}
\newcommand{\gG}{{\mathcal G}}

\newcommand{\iI}{\mathcal{I}}

\newcommand{\lL}{\mathcal{L}}

\newcommand{\vV}{{\mathcal{V}}}
\newcommand{\cC}{\mathcal{C}}
\newcommand{\pP}{\mathcal{P}}

\newcommand{\Cg}{\mathrm{Cg}}
\newcommand{\Clo}{\mathrm{Clo}}
\newcommand{\Con}{\mathrm{Con}}

\newcommand{\vtt}[2]{\begin{pmatrix} #1\\ #2\end{pmatrix}}

\newcommand{\algop}[2]{\langle {#1}, {#2} \rangle}

\newcommand{\setsuchthat}{\ : \ }

\newcommand{\meet}{\wedge}
\newcommand{\join}{\vee}

\newcommand{\comment}[1]{\textsf{[#1]}}

\date{\today}

\begin{document}


\title{Finiteness properties of direct products of algebraic structures}

\author{Peter Mayr} 
\address[Peter Mayr]{Department of Mathematics, CU Boulder, USA}
\email{peter.mayr@colorado.edu}
\author{Nik Ruskuc}
\address[Nik Ru\v{s}kuc]{School of Mathematics and Statistics, University of St Andrews, St Andrews, Scotland, UK}
\email{nik.ruskuc@st-andrews.ac.uk}

\thanks{The first author was supported by the Austrian Science Fund (FWF): P24285
 and the National Science Foundation under Grant No. DMS 1500254}

\begin{abstract}
We consider the preservation of properties of being finitely generated, being finitely presented and being residually finite under direct products in the context of different types of algebraic structures.
The structures considered include Mal'cev algebras (including groups, rings and other classical algebras, as well as loops), idempotent algebras (including lattices), semigroups, and algebras in congruence modular varieties.
We aim to identify as broad classes as possible in which the `expected' preservation results ($\A\times\B$ satisfies property $\pP$ if and only if $\A$ and $\B$ satisfy $\pP$) hold, and to exhibit ways in which they may fail outside those classes.
\end{abstract}

\maketitle

\section{Introduction: direct products in algebra}
\label{secIntro}

The direct product is arguably the most elementary and ubiquitous construction in algebra.
Because of its uncomplicated relationship with the constituent factors, the `natural' statements of the form
\begin{quotation}
\it
$\A\times\B$ satisfies property $\pP$ if and only if $\A$ and $\B$ both satisfy $\pP$
\end{quotation}
abound, and are often taken as read. For example, in the class of groups, the above statement is true for 
$\pP$ being any of the following: finitely generated, finitely presented, residually finite, having decidable word problem,
 locally finite, torsion, solvable, nilpotent, \dots 
Somewhat curiously, the statement does not hold for the property of being hopfian \cite{corner65}.
In fact, these statements hold in the wider class of monoids (excluding of course the group-specific properties  of being nilpotent or solvable). However, if one moves on to the class of semigroups, the landscape begins to change and become more interesting. Specifically, the direct product $\S\times\T$ of two infinite semigroups is:
\begin{itemize}
\item
finitely generated if and only if $\S$ and $\T$ are finitely generated and have no indecomposable elements \cite[Theorem 2.1]{Robertson98};
\item
finitely presented if and only if both $\S$ and $\T$ are finitely presented, have no indecomposable elements, and satisfy an additional condition called stability \cite[Theorem 3.5]{Robertson98};
\item
residually finite if and only if both $\S$ and $\T$ are residually finite, but for apparently non-trivial reasons \cite[Theorem 1]{Gray09}.
\end{itemize}
To add further interest, it remains an open question whether it is possible to algorithmically determine if the direct product of two finitely presented semigroups  (given by their finite presentations) is finitely presented.
By way of contrast, the question where one factor is known to be finite is decidable (while still non-trivial), see \cite[Theorem 1.2]{Araujo00}.

The purpose of this article is to raise this problematic to the level of general algebra, and initiate a comparative study across different classes of algebras. We will focus on three well known properties of being
\begin{itemize}
\item
finitely generated,
\item
finitely presented, and
\item
residually finite,
\end{itemize}
and ask under which conditions does the direct product of two algebras $\A$ and $\B$ from a certain class of algebras $\cC$ have one of these properties.
The classes $\cC$ we wish to cover are many, and we will resort to the language and methodology of universal algebra to enable us to state results for several classes at once. Here is the list of classes considered here:
\begin{itemize}
\item
groups;
\item
other `classical' structures: rings, (associative and non-associative) algebras, modules;
\item
group-like, non-associative algebras: loops and quasigroups;
\item
Mal'cev algebras (encompassing all of the above);
\item
semigroups and monoids;
\item
lattices;
\item
idempotent algebras (including lattices);
\item
algebras in congruence modular varieties (including Mal'cev algebras and lattices).
\end{itemize}

 The main general preservation results we prove concern finite generation in
 congruence permutable or idempotent varieties (Theorems~\ref{fg:Mal'cev} and~\ref{th:Gidempotent}), and residual
 finiteness in congruence modular varieties (Theorem~\ref{th:RCM}). Contrasting these are some rather
 surprising negative results, such as: 
\begin{itemize}
\item
 There exist expansions of groups (and expansions of lattices) $\A,\B$ that are finitely generated
 without $\A\times\B$ being finitely generated (Remarks~\ref{re:Zx} and~\ref{re:Ns}).
\item
 There exist algebras $\A,\B$ with $\B$ not finitely presented such that $\A\times\B$ is finitely presented
 (Example~\ref{ex:Gset}).
\item
 Direct products of two infinite loops (or two infinite lattices) are never finitely presented
 (Theorems~\ref{th:Ploop} and~\ref{th:Lfp}).
\end{itemize}

Due to the great number of different settings that are considered here, we will not attempt a systematic definition of all the objects, concepts and properties. Instead, we will implicitly rely on the reader's prior, if not explicit or systematic, familiarity with many of the elementary notions used, and will confine ourselves to giving an outline of the basic definitions and facts before each individual result necessary to understand its statement and proof. 
For a more systematic introduction to universal algebra we refer the reader to \cite{burris81,mckenzie87}, 
\cite{hungerford80} for classical algebraic structures,  \cite{howie95} for semigroups and monoids, \cite{bruck58} for loops and quasigroups, \cite{Gr:LTF} for lattices.

Throughout the paper we will denote algebras (i.e. sets endowed by certain operations) by boldface capital letters $\A$, $\B$, $\C$,\dots; the underlying sets of elements will be denoted by the corresponding plain letters $A$, $B$, $C$,\dots. 
 A \emph{type} of algebras is a set $\fF$ of function symbols with a non-negative integer (the \emph{arity})
 associated to each $f\in\fF$. An algebra $\A$ of type $\fF$ is a pair
 $\algop{A}{\{ f^\A \setsuchthat f\in\fF \}}$ where $f^\A$ is an operation on $A$ whose arity equals 
 the arity of the symbol $f$.

 We define \emph{terms} of type $\fF$ with variables $x_1,x_2,\dots$ and their length inductively:
 every variable $x_1,x_2,\dots$ is a term of length $1$. For any $k$-ary $f\in\fF$ and terms $s_1,\dots,s_k$ we have
 that $f(s_1,\dots,s_k)$ is a term of length $1+|s_1|+\dots+|s_k|$. For a term $t$ and an algebra $\A$ of type $\fF$
 let $t^\A$ denote its induced \emph{term operation} on $\A$.
 If there is no danger of confusion, we simply write $t$ instead of $t^\A$.

For two algebras $\A$, $\B$ of the same type $\fF$, their direct product $\C=\A\times\B$
has the carrier set $C=A\times B$ and componentwise operations
\[
f^{\A\times\B}((a_1,b_1),\dots,(a_k,b_k))=(f^\A(a_1,\dots,a_k),f^\B(b_1,\dots,b_k)),
\]
where $f\in\fF$ is a $k$-ary operation symbol, $a_1,\dots, a_k\in A$ and $b_1,\dots,b_k\in B$.
Occasionally we will resort to the vertical notation for pairs, writing $\vtt{a}{b}$ for $(a,b)$.
Associated with $\A\times \B$ are the \emph{natural projections} $\pi_A :A\times B\rightarrow A$, $(a,b)\mapsto a$ and $\pi_B :A\times B :\rightarrow B$, $(a,b)\mapsto b$, which are epimorphisms. In particular, $\A$ and $\B$ are homomorphic images of $\A\times\B$.
In addition, often, but not always, $\A$ and $\B$ naturally embed into $\A\times \B$; this is certainly the case with structures with a neutral element (such as groups, monoid, rings, etc.) and idempotent algebras.

 We denote \emph{varieties} (i.e., classes of algebras of the same type that are defined by identities)
 by calligraphic letters $\vV, \gG$,\dots For example, the variety of groups $\gG$ consists
 of all algebras of type $\fF = \{\cdot,^{-1},1\}$ with arities $2,1,0$ defined by the identities
\[ (x\cdot y)\cdot z = x\cdot (y\cdot z),\  x\cdot 1 = x,\ x\cdot x^{-1} = 1. \]

For an algebra $\A$ we denote the set of all congruences on $\A$ by $\Con(\A)$, the trivial congruence (equality)
 by $0_A$ and the total congruence by $1_A$.
The congruence on $\A$ generated by a set $R\subseteq A\times A$ is denoted by $\Cg_\A(R)$.

 A variety $\vV$ is \emph{congruence permutable} or \emph{Mal'cev} if for every algebra $\A$ in $\vV$ all
 congruences permute, i.e., 
\[ \forall\alpha,\beta\in\Con(\A)\colon \alpha\circ\beta = \beta\circ\alpha. \]
 Equivalently $\vV$ is congruence permutable if and only if it has a ternary term $m$ such that the identities
\[
m(x,y,y)= x,\ m(y,y,x) = x
\]
 hold in $\vV$. Mal'cev varieties contain in particular all varieties that have the operations of a group (multiplication, inversion, identity) or quasigroup (multiplication, left division, right division).

 A variety $\vV$ is \emph{congruence modular} if every algebra in $\vV$ has a modular congruence lattice.
 Congruence modular varieties contain Mal'cev varieties and all varieties with lattice operations. 

\emph{Idempotent algebras} are another general class that will play a prominent role throughout. 
We say that an algebra $\A$ of type $\fF$ is \emph{idempotent} if 
\[
f^\A(x,x,\dots,x)=x
\]
for all $f\in \fF$ and all $x\in A$. Note that this implies that the above equality in fact holds for all the terms over  $\fF$.
One distinguished example of idempotent algebras is provided by lattices, considered as algebras of type $(\meet,\join)$.

\section{Finite generation}
\label{secfg}

In this section we ask under which conditions is the direct product $\A\times \B$ of two algebras finitely generated. We immediately note that $\A$ and $\B$ are homomorphic images of $\A\times \B$, and so we have:

\begin{obs}
\label{obs1}
If $\A\times\B$ is finitely generated, then so are both $\A$ and $\B$.
\end{obs}

So the only question to consider is which conditions ensure that $\A$ and $\B$ finitely generated implies that $\A\times\B$ is finitely generated. The statement certainly holds for groups, and it actually turns out that this can be generalised to Mal'cev varieties.

\begin{thm}
\label{fg:Mal'cev}
Let $\vV$ be a Mal'cev variety of type $\fF$, and let $\A,\B\in\vV$.
Suppose that there exist $a_0\in A$ and $b_0\in B$ such that the sets
$\{  f^\A(a_0,\dots,a_0)\setsuchthat f\in\fF\}$ and
$\{  f^\B(b_0,\dots,b_0)\setsuchthat f\in\fF\}$ are finite.
Then $\A\times\B$ is finitely generated if and only if both $\A$ and $\B$ are finitely generated.
\end{thm}

\begin{proof}
This result was known to A. Geddes (unpublished).
Only the converse direction needs to be proved. Let
$\A,\B$ be respectively generated by finite sets $X,Y$, and
let
\[
 X^\prime=X\cup\{ f^\A(a_0,\dots,a_0)\setsuchthat f\in\fF \},\ 
Y^\prime=Y\cup\{ f^\B(b_0,\dots,b_0)\setsuchthat f\in\fF \},
\]
which are finite sets by the assumptions in the theorem.
 We claim that $\A\times\B$ is generated by
the finite set
\[
Z=(X^\prime\times \{b_0\})\cup (\{a_0\}\times Y^\prime)\cup \{(a_0,b_0)\}.
\]

 First we prove
\begin{equation} \label{eq:svs}
 A\times\{b_0\} \subseteq \langle Z\rangle.
\end{equation}
Let $a\in A$ be arbitrary. Since $\A$ is generated by $X$, supposing $X=\{x_1,\dots,x_m\}$,
we have an $m$-ary term $s$ over $\fF$
such that $s^{\A}(x_1,\dots,x_m) = a$.
 Now we can use induction on the length of $s$. 
 If $s$ is a variable, then $a\in X$ and $(a,b_0)\in Z$ by definition.
 So assume that $s = f(t_1,\dots,t_k)$ for some $k$-ary $f\in\fF$ and terms $t_1,\dots,t_k$.
 Let $a_i = t_i^\A(x_1,\dots,x_m)$ for $1\leq i\leq k$.
 By the induction hypothesis we have $(a_1,b_0),\dots,(a_k,b_0)\in\langle Z\rangle$.
 From this and the definition of $Z$ it follows that the pairs
\[
 \vtt{f^\A(a_1,\dots,a_k)}{f^\B(b_0,\dots,b_0)},
\vtt{a_0}{f^\B(b_0,\dots,b_0)} ,
 \vtt{a_0}{b_0} 
\]
all belong to $\langle Z\rangle$ (in fact, the second and third belong to $Z$).
 By applying the Mal'cev term $m$ to the three pairs above, we obtain
 $(f^\A(a_1,\dots,a_k),b_0)\in\langle Z\rangle$ which proves~\eqref{eq:svs}.

 It follows similarly that 
\[
 \{a_0\}\times B \subseteq\langle Z\rangle.
\]
 Now for an arbitrary $(a,b)\in A\times B$ we have that
\[
 \vtt{a}{b_0}  ,
 \vtt{a_0}{b_0}  ,
 \vtt{a_0}{b}  
\]
all belong to $\langle Z\rangle$.
 Applying the Mal'cev term $m$ once more yields $(a,b)\in\langle Z\rangle$.
\end{proof}

\begin{rem} \label{re:Zx}
The extra assumption that the sets  
$\{  f^\A(a_0,\dots,a_0)\setsuchthat f\in\fF\}$ and
$\{  f^\B(b_0,\dots,b_0)\setsuchthat f\in\fF\}$ 
be finite is essential.
Consider for example the algebra $\A$ 
of type $(+,-,c_a\ (a\in A))$, where $\algop{A}{+,-}$ is an abelian group that is not finitely generated, and $c_a$ is a constant symbol representing the element $a$ of $A$.
Then, clearly, $\A$ is finitely generated (by the empty set, say), and we claim that
the direct square $\A\times \A$ is not finitely generated.
Suppose that $\A\times \A$ is generated by a finite set $Z$.
This means that considered just as an abelian group (i.e. without the constants) it would be generated by the set $Z_1=Z\cup \{(a,a)\setsuchthat a\in A\}$.
It is straightforward to show that the subgroup generated by $Z_1$ to
$\{ (a,b)\in A\times A\setsuchthat a-b\in A^\prime\}$, where $A^\prime$ is the subgroup of $A$ generated by the set $\{ x-y\setsuchthat (x,y)\in Z\}$.
Since $Z$ is finite and $\algop{A}{+,-}$ is not finitely generated it follows that $A^\prime\neq A$, and hence $Z$ does not generate the entire $\A\times\A$, a contradiction.
\end{rem}

The extra assumption is sufficiently weak for Theorem \ref{fg:Mal'cev}
to hold in any Malcev varieties with finitely many basic operations, or those in which every member contains an idempotent. In particular:

\begin{cor}
Let $\vV$ be any of the following classes: groups, rings, 
modules over a ring, Lie algebras,
loops, or quasigroups, and let $\A,\B\in\vV$.
Then $\A\times\B$ is finitely generated if and only if both $\A$ and $\B$ are finitely generated.
\end{cor}

Another large class for which the if and only if result holds is provided by idempotent algebras.
In fact we will prove a slightly more general result.
Note that the idempotency condition can be interpreted as the requirement for the unary clone of the algebra to consist only of the identity mapping. Recall that the $k$-ary clone $\Clo_k(\A)$ is obtained by taking the set of all 
$k$-ary terms over the type of $\A$, and interpreting them as $k$-ary functions $A^k\rightarrow A$;
for details see \cite{mckenzie87}.

\begin{thm}
\label{th:Gidempotent}
 Let $\A,\B$ be algebras of the same type.
 Assume that all functions in $\Clo_1(\A)$ and $\Clo_1(\B)$ are surjective, and that,
 moreover, one of $\Clo_1(\A)$ or $\Clo_1(\B)$ is finite (and hence a group of bijections).
Then $\A\times\B$ is finitely generated if and only if both $\A$ and $\B$ are finitely generated.
 In particular, the direct product of two idempotent algebras is finitely generated if and only if both algebras are finitely generated.
\end{thm}

\begin{proof}
Only the converse direction needs to be proved.
Let $X=\{x_1,\dots,x_m\}$ and $Y=\{y_1,\dots,y_n\}$ be finite generating sets for $\A$ and $\B$ respectively.
Without loss of generality assume that $\Clo_1(\A)$ is a finite group, and that $X$ is closed under its action.
We will prove that $Z=X\times Y$ is a generating set for $\A\times\B$.
Let $a\in A$, $b\in B$ be arbitrary.
Since $X$ generates $\A$, there exists an $m$-ary term $p$ such that
\[
a=p^\A (x_1,\dots,x_m).
\]
Denote by $p_1$ the unary term $p(x,\dots,x)$.
By assumption the unary function $p_1^\B (x)$ is surjective, and so there exists $b^\prime\in B$ such that
\[
p_1^\B(b^\prime)=b.
\]
Since $Y$ generates $\B$ there exists an $n$-ary term $q$ such that
\[
b^\prime=q^\B (y_1,\dots,y_n).
\]
Let $q_1$ be the unary term $q(x,\dots,x)$, and let
\[
x_i^\prime = (q_1^\A)^{-1}(x_i)\in X\ \ (i=1,\dots,m).
\]
Now, for each $i=1,\dots,m$ we have
\[
\vtt{x_i^\prime}{y_1},\dots,\vtt{x_i^\prime}{y_n}\in Z,
\]
from which it follows that 
\[
 q^{\A\times\B} ( \vtt{x_i^\prime}{y_1},\dots,\vtt{x_i^\prime}{y_n} )
=\vtt{q_1^\A(x_i^\prime)}{q^\B(y_1,\dots,y_n)}=\vtt{x_i}{b^\prime}
\]
 is in $\langle Z\rangle$.
Now applying $p$ to
\[
\vtt{x_1}{b^\prime},\dots,\vtt{x_m}{b^\prime}\in\langle Z\rangle
\]
we obtain
\[ 
p^{\A\times\B}(\vtt{x_1}{b^\prime},\dots,\vtt{x_m}{b^\prime})
=\vtt{p^\A(x_1,\dots,x_m)}{p_1^\B(b^\prime)}=\vtt{a}{b}
\in \langle Z\rangle 
\]
 as required.
\end{proof}

\begin{cor}
\label{corlatfg}
The direct product $\A\times \B$ of two lattices is finitely generated if and only if $\A$ and $\B$ are finitely generated.
\end{cor}

\begin{rem} \label{re:Ns}
The above corollary does not extend to arbitrary 
 expansions of lattices (algebras with near-unanimity term or
 algebras in congruence distributive varieties)
 in general. Indeed, 
the algebra $\algop{\N}{\max,\min, s}$, where
$s:x\mapsto x+1$ is the successor function, is clearly generated by $1$.
We claim that its square is not finitely generated.
Suppose otherwise, and let $Z$ be a finite generating set.
Noting that if $(x,y)=m((x_1,y_1),(x_2,y_2))$, where $m$ is either of $\max$ or $\min$,
we have $|x-y|\leq \max(|x_1-y_1|,|x_2-y_2|)$,
while if $(x,y)=s(x_1,y_1)$ then $|x-y|=|x_1-y_1|$,
it follows that $\langle Z\rangle$ is contained in
$\{ (x,y)\setsuchthat |x-y|\leq M\}$, where $M=\max\{ |z-u|\setsuchthat (z,u)\in Z\}$,
which is clearly not the entire $\N\times\N$, a contradiction.
This example also shows that Theorem \ref{fg:Mal'cev} does not generalize to congruence modular varieties.
However, Theorem \ref{th:Gidempotent} does apply to lattices with involutions.
\end{rem}

\begin{rem}
The assumption that at least one of the two clones is finite in Theorem \ref{th:Gidempotent} is necessary, even if both consist entirely of bijections:
for instance, $\Z\times\Z$ is not finitely generated when considered  as a $\Z$-set.
Likewise, finiteness of unary clones is not sufficient for finite generation:
 Consider the free semigroup $\S$ in the variety defined by $x^2=0$ on $3$ generators.
It is well known that $\S$ is infinite; see \cite[Chapter 2]{lothaire97}.
 Then $\S\times \S$ is not finitely generated by \cite[Theorem 2.1]{Robertson98}, even though $\Clo_1(\S)$ contains only two mappings, namely the identity and the constant $0$-mapping. 
\end{rem}

\section{Finite presentability}
\label{secfp}

To be \emph{finitely presented} means to be isomorphic to a quotient of a finitely generated free algebra by a finitely generated congruence.
Compared with finite generation and residual finiteness, finite presentability presents a subtle issue:
to be finitely presented is \emph{not} an intrinsic property of the algebra, but rather it also makes reference to the class within which the free algebras are taken (typically a variety). For instance, a finitely generated free group is (obviously) finitely presented as a group, is also finitely presented as a monoid, but is not finitely presented as a magma (groupoid, a set with an arbitrary binary operation). For a variety $\vV$ of type $\fF$ we will denote by $ F_\vV(X)$ the \emph{free algebra} over $X$ in $\vV$. It consists of all terms of type $\fF$ over variables $X$ modulo the identities that hold in $\vV$.

\subsection{From products to factors}

 Unlike for finite generation, the fact that a direct product is finitely presented
 does not always imply that the factors are finitely presented as demonstrated by the following example
 the main idea of which is due to Keith Kearnes and \'Agnes Szendrei.

\begin{exa} \label{ex:Gset}
 Let $\G$ be the free group on two generators $x,y$. 
Then $\G$ is a semidirect product of $A = \langle y\rangle$
 and a normal subgroup $B = \langle x^a \setsuchthat a\in A\rangle$.

 We let $\G$ act on itself by multiplication on the right to obtain a regular $\G$-set. Since as a group $\G$ is
 generated by $x$ and $y$, this regular $\G$-set is term-equivalent to the algebra $\C = \algop{G}{f_x,g_x,f_y,g_y}$ where
 for $z\in G$ we have
\[ f_x(z) = zx,\ g_x(z) = zx^{-1},\ f_y(z) = zy,\ g_y(z) = zy^{-1}. \]
 The cosets of the subgroups $A$ and $B$ of $\G$ yield congruences
\begin{align*}
 \alpha & = \{ (u,v) \in C^2 \setsuchthat Au= Av \}, \\
 \beta & = \{ (u,v) \in C^2 \setsuchthat Bu = Bv \}
\end{align*}
 on $\C$. 
 We claim:
\begin{enumerate}
\item \label{it:VC}
 The algebra $\C$ is free in the variety $\vV$ of unary algebras of type $\{f_x,g_x,f_y,g_y\}$ defined by the identities
\[
g_xf_x(z)=z,\ f_xg_x(z)=z,\ g_yf_y(z)=z,\ f_yg_y(z)=z.
\]
\item \label{it:CaCb}
 $\C$ is isomorphic to $\C/\alpha\times \C/\beta$.
\item \label{it:bnfg}
 $\C/\beta$ is not finitely presented. 
\end{enumerate}
 For~\eqref{it:VC}, clearly $\C$ belongs to $\vV$ and is one-generated.
From the defining identities of $\vV$, it follows that $ F_\vV(z)$ is isomorphic to $\C$.
For~\eqref{it:CaCb} we note that $\alpha\wedge\beta = 0_C$
and $\alpha\circ\beta = \beta\circ\alpha = 1_C$ because $A\cap B=\{1\}$ and
$AB=BA=G$.
Finally for~\eqref{it:bnfg} it suffices to show that $\beta$ is not finitely generated as a congruence of
 $\C$. Suppose otherwise that $\beta$ is generated by $(u_1,v_1),\dots,(u_k,v_k)$ for some $k\in\N$,
 $u_1,\dots,u_k,v_1,\dots,v_k\in G$. Then it is straightforward that $B$ is generated as a group by $u_1v_1^{-1},\dots,u_kv_k^{-1}$.
 But $B$ is  a normal subgroup of a free group of infinite index, and so is not finitely generated by \cite[Proposition 1.3.12]{lyndon01}. Hence $\beta$
 is not finitely generated, and the direct factor $\C/\beta$ is not finitely presented.
\end{exa}

To prove the implication from $\A\times\B$ being finitely presented to $\A$ and $\B$ being finitely presented it is sufficient to prove that the kernels of natural projections $\pi_A$, $\pi_B$ are finitely generated congruences on $\A\times\B$.
 This is the case for a wide range of varieties. For instance:

\begin{thm} \label{th:fgVfp}
 Let $\vV$ be a variety in which every direct product of two finitely generated algebras is finitely generated, and let $\A,\B\in\vV$.
 If $\A\times \B$ is finitely presented, then both $\A$ and $\B$ are finitely presented.
\end{thm}

\begin{proof}
 Assume $\A\times\B$ is finitely presented.
 The kernel of the projection $\pi_A$ is $0_A\times 1_B$ which can be regarded as a subalgebra of $\A^2\times \B^2$.
 As such it is isomorphic to $\A\times\B^2$. Since $\A$ and $\B$ are finitely generated,
 by the assumption on $\vV$ also $\A\times\B^2$ will be finitely generated. Hence $0_A\times 1_B$ is finitely
 generated as a congruence. Thus $\A$ and similarly $\B$ are finitely presented.
\end{proof}




 In particular Theorem~\ref{th:fgVfp} applies to idempotent and Mal'cev varieties as we saw in Section~\ref{secfg}.
 We also recall that finite generation is not preserved by direct products in all congruence modular varieties. However we will prove
 the analog of Theorem~\ref{th:fgVfp} in these varieties as well.
 For that
 we need the following two auxiliary results.

\begin{lem} \label{le:rs}
 Let $\A$ and $\B$ be algebras in a variety $\vV$, and let $\rho\in\Con(\A\times\B)$.
\begin{enumerate}
\item\label{it:ts}
 The congruence $\rho\vee (0_A\times 1_B)$ is a product congruence, namely $\rho\vee (0_A\times 1_B) = \tau\times 1_B$ for
\[ \tau = \{ (u,v)\in A^2 \setsuchthat (u,b)\equiv_{\rho\vee (0_A\times 1_B)} (v,c) \text{ for all } b,c \in B \}. \] 
\item\label{it:rs}
 If $\vV$ is congruence modular, then $\rho\wedge (1_A\times 0_B) = \sigma\times 0_B$ for the congruence
\[ \sigma = \{ (u,v)\in A^2 \setsuchthat (u,b)\equiv_\rho (v,b) \text{ for some } b\in B \} \]
 on $\A$. 
\end{enumerate}
\end{lem}

\begin{proof}
 For \eqref{it:ts} we first prove that the equality 
\begin{equation} \label{eq:ts}
 \rho\vee (0_A\times 1_B) = \tau\times 1_B
\end{equation}
 holds. Let $u,v\in A, r,s\in B$. Assume $(u,r)\equiv_{\rho\vee (0_A\times 1_B)} (v,s)$.
 Then $(u,b)\equiv_{\rho\vee (0_A\times 1_B)} (v,c)$ for all $b,c\in B$. Hence $(u,v)\in\tau$ and 
 $((u,r),(v,s))\in\tau\times 1_B$. Conversely, if $((u,r),(v,s))\in\tau\times 1_B$, then
 $(u,r)\equiv_{\rho\vee (0_A\times 1_B)} (v,s)$ by the definition of $\tau$. Thus~\eqref{eq:ts} is proved.

 It remains to check that $\tau$ is a congruence.
 By~\eqref{eq:ts} we have that $\tau$ is the projection of $\rho\vee (0_A\times 1_B)$ onto $A$.
 The projection of any congruence is preserved by the operations of $\A$, is reflexive and symmetric.
 To see that $\tau$ is transitive, we use the particular form of $\rho\vee (0_A\times 1_B)$.
 Let $(u,v)\in\tau$ and $(v,w)\in\tau$. Then $(u,b)\equiv_{\rho\vee (0_A\times 1_B)} (v,c)$ for all
 $b,c\in B$ and $(v,d)\equiv_{\rho\vee (0_A\times 1_B)} (w,e)$ for all $d,e\in B$.
 So $(u,b)\equiv_{\rho\vee (0_A\times 1_B)} (w,e)$ for all $b,e\in B$ and $(u,w)\in\tau$.
 Thus $\tau\in\Con(\A)$ and item~\eqref{it:ts} is proved.

 For~\eqref{it:rs} assume that $\vV$ is congruence modular. Let $\mu = \rho\wedge (1_A\times 0_B)$. Then
\begin{align*}
 \mu & = \mu \vee [(1_A\times 0_B) \wedge (0_A\times 1_B)], \\
 & = (1_A\times 0_B) \wedge [ \mu \vee (0_A\times 1_B)] \text{ by the modular law.}
\end{align*}
 By~\eqref{it:ts} we have $\sigma\in\Con(\A)$ such that $\mu \vee (0_A\times 1_B) = \sigma\times 1_B$.
 It follows that
\[ \mu = \sigma\times 0_B. \]
 So $\sigma$ is the projection of $\mu$ onto $A$, i.e.,
\[ (u,v)\in\sigma \text{ iff } (u,b) \equiv_\mu (v,b) \text{ for some } b\in B. \]
 Since $\mu = \rho\wedge (1_A\times 0_B)$, the latter condition is equivalent to $(u,b)  \equiv_\rho (v,b)$
 for some $b\in B$. Hence $\sigma$ satisfies~\eqref{it:rs}.
\end{proof}

\begin{lem} \label{le:CMfg}
 Let $\vV$ be a congruence modular variety of type $\fF$, and let $\A,\B\in\vV$.
Suppose that there exist $a_0\in A$ and $b_0\in B$ such that the sets
$\{  f^\A(a_0,\dots,a_0)\setsuchthat f\in\fF\}$ and
$\{  f^\B(b_0,\dots,b_0)\setsuchthat f\in\fF\}$ are finite.
 If $\A$ is finitely generated, then the kernel of the projection of $A\times B$ onto $B$
 is a finitely generated congruence of $\A\times\B$.
\end{lem} 

\begin{proof}
 Let $X = \{x_1,\dots,x_k\}$ be a finite generating set for $\A$, and define
\begin{multline*}
\rho = \Cg_{\A\times\B}\bigl( \bigl\{ ((x,b_0),(a_0,b_0)) \setsuchthat x\in X \bigr\} 
\\
\cup \bigl\{ ((f^\A(a_0,\dots,a_0),b_0),(a_0,b_0)) \setsuchthat f\in\fF \bigr\} \bigr). 
\end{multline*}
 Clearly $\rho\leq 1_A\times 0_B$. By Lemma~\ref{le:rs}~\eqref{it:rs} we have $\rho = \sigma \times 0_B$
 for $\sigma\in\Con(\A)$ with 
\begin{equation} \label{eq:scg}
 \sigma = \Cg_\A\bigl( \bigl\{ (x,a_0) \setsuchthat x\in X \bigr\} \cup \bigl\{ (f^\A(a_0,\dots,a_0),a_0) \setsuchthat f\in\fF \bigr\} \bigr). 
\end{equation}
 Since $\rho$ is finitely generated, the result will follow once we have shown that $\sigma = 1_A$.
 For that we claim
\begin{equation} \label{eq:usa}
 a \equiv_\sigma a_0 \text{ for all } a\in A.
\end{equation}
 Since $X$ generates $\A$, we have a $k$-ary term $t$ such that $t^\A(x_1,\dots,x_k) = a$.
 We prove~\eqref{eq:usa} by induction on the length of $t$. If $t$ is a variable, the statement
 is true by~\eqref{eq:scg}. So assume $t=f(s_1,\dots,s_\ell)$ for $f$ an $\ell$-ary operation
 in $\fF$ and $k$-ary terms $s_1,\dots,s_\ell$. We obtain
\begin{align*}
 a & = f^\A(s^\A_1(x_1,\dots,x_k),\dots,s^\A_\ell(x_1,\dots,x_k)) \span\omit\span\omit \\
 & \equiv_\sigma f^\A(a_0,\dots,a_0) &&\text{(by the induction assumption)} \\
 & \equiv_\sigma a_0  &&\text{(by~\eqref{eq:scg})}.
\end{align*}
 This proves~\eqref{eq:usa}. Thus $\sigma = 1_A$ and $\rho = 1_A\times 0_B$ is finitely generated.
\end{proof}

\begin{thm} \label{th:CMfp}
 Let $\vV$ be a congruence modular variety of type $\fF$, and let $\A,\B\in\vV$.
Suppose that there exist $a_0\in A$ and $b_0\in B$ such that the sets
$\{  f^\A(a_0,\dots,a_0)\setsuchthat f\in\fF\}$ and
$\{  f^\B(b_0,\dots,b_0)\setsuchthat f\in\fF\}$ are finite.
 If $\A\times\B$ is finitely presented, then both $\A$ and $\B$ are finitely presented.
\end{thm}

\begin{proof}
 Assume  $\A\times\B$ is finitely presented in a congruence modular variety of finite type, i.e. it is a quotient of a finitely generated free algebra in $\vV$ by a finitely generated congruence. 
 In particular $\A$ and $\B$ are finitely generated, and hence, by Lemma~\ref{le:CMfg},
are quotients of $\A\times\B$ by finitely generated congruences. 
Therefore both $\A$ and $\B$ are finitely presented.
\end{proof}

Of course, as in Section \ref{secfg}, Theorem \ref{th:CMfp} in particular applies to congruence modular varieties of finite type, or those whose members contain idempotents.
We also remark that the analog for Theorems~\ref{th:fgVfp} and~\ref{th:CMfp}
 also holds for semigroups; see \cite{Robertson98}.

\subsection{From factors to products}

We first make a general observation.

\begin{prop}   \label{prop:Pfree}
 For any variety $\vV$ the following are equivalent:
\begin{enumerate}
\item \label{it:free}
 The direct product of any two finitely generated free algebras in $\vV$ is finitely presented.
\item \label{it:fp}
 The direct product of any two finitely presented algebras in $\vV$ is finitely presented.
\end{enumerate}
\end{prop}

\begin{proof}
 The implication \eqref{it:fp}$\Rightarrow$\eqref{it:free} is immediate.

 For proving the converse, assume~\eqref{it:free}.
 Note that this implies
 in particular that the direct product of any two finitely generated algebras in $\vV$ is finitely generated.
 Let $\A,\B\in\vV$ be finitely presented. Then we have a finite set $X = \{x_1,\dots,x_m\}$ and a finite
 set $R$ of pairs of terms over $X$ such that $\A \cong F_\vV(X)/\Cg_{F_\vV(X)}(R)$.
 Without loss of generality, we may assume that $R = R^{-1}$ and
 $\{(x_i,x_i)\setsuchthat 1\leq i\leq m\} \subseteq R$.
 Similarly $\B \cong F_\vV(Y)/\Cg_{F_\vV(Y)}(S)$ for finite $Y$ and $S$.
 By assumption $F_\vV(X)\times F_\vV(Y)$ is finitely presented.
 So $\A\times\B$ being finitely presented is equivalent to
 $\Cg_{F_\vV(X)}(R)\times\Cg_{F_\vV(Y)}(S)$ being finitely generated as a congruence of $F_\vV(X)\times F_\vV(Y)$.

 We claim that
\begin{equation} \label{eq:Rx0}
 \Cg_{F_\vV(X)}(R)\times 0_{F_\vV(Y)} \text{ is finitely generated.}
\end{equation} 
 First observe that $\langle R\rangle\times 0_{F_\vV(Y)}\leq F_\vV(X)^2\times F_\vV(Y)^2$
 is finitely generated as direct product of finitely generated algebras. Let $R'$ be a finite generating set.

 Assume that $(u,v)\in\Cg_{F_\vV(X)}(R)$. By our assumptions $\langle R\rangle$ is reflexive and symmetric on $F_\vV(X)$.
 So $\Cg_{F_\vV(X)}(R)$ is the transitive closure of $\langle R\rangle$. We have $w_0,\dots,w_n\in F_\vV(X)$
 such that
\begin{equation*}
 u = w_0, v = w_n, (w_i,w_{i+1}) \in\langle R\rangle
\end{equation*}
 for all $i \in\{0,\dots,n-1\}$. 
It then follows that 
\[
(\vtt{w_i}{t},\vtt{w_{i+1}}{t})\in
\langle R\rangle\times 0_{F_\vV(Y)}
\]
for any $t\in F_\vV(Y)$, and hence, since $R'$ is a generating set for
$\langle R\rangle\times 0_{F_\vV(Y)}$, we have
\begin{equation*}
 (\vtt{w_i}{t},\vtt{w_{i+1}}{t}) \in\langle R'\rangle
\end{equation*}
for all $i \in\{0,\dots,n-1\}$.
 Hence $(\vtt{u}{t},\vtt{v}{t})$ is contained in the transitive closure of
 $\langle R'\rangle$, that is, in $\Cg_{F_\vV(X)\times F_\vV(Y)}(R')$. Thus $\Cg_{F_\vV(X)}(R)\times 0_{F_\vV(Y)}$ is generated
 by $R'$, and~\eqref{eq:Rx0} is proved. 
 
 Similarly $0_{F_\vV(X)}\times\Cg_{F_\vV(Y)}(S)$ is finitely generated. Since $\Cg_{F_\vV(X)}(R)\times\Cg_{F_\vV(Y)}(S)$
 is the join of $\Cg_{F_\vV(X)}(R)\times 0_{F_\vV(Y)}$ and $0_{F_\vV(X)}\times\Cg_{F_\vV(Y)}(S)$, the result follows.
\end{proof}

As an immediate consequence, we obtain the if and only if statement for all the `classical' algebraic structures:

\begin{cor} \label{cor:class}
Let $\vV$ be any of the following varieties: groups, rings, modules over a ring $R$, algebras over a field,
 Lie algebras, or monoids, and let $\A,\B\in\vV$. Then $\A\times\B$ is finitely presented if and only if $\A$ and $\B$ are finitely presented.
\end{cor}

\begin{proof}
 Let $X,Y$ be disjoint finite sets. Then $F_\vV(X)\times F_\vV(Y)$ is isomorphic to $ F_\vV(X\cup Y)/\rho$ where
 the congruence $\rho$ is generated by the following set $R$.

\begin{center}
\begin{tabular}{|c|c|} \hline
 $\vV$ & $R$ \\ \hline
 groups, monoids & $\{ (xy,yx) \setsuchthat x\in X, y\in Y \}$ \\
 modules & $\emptyset$ \\
 rings, algebras & $\{ (xy,0), (yx,0) \setsuchthat x\in X, y\in Y \}$ \\
 Lie algebras & $\{ ([x,y],0), ([y,x],0) \setsuchthat x\in X, y\in Y \}$ \\ \hline
\end{tabular}
\end{center}
 Since $\rho$ is finitely generated in each case, $F_\vV(X)\times F_\vV(Y)$ is finitely presented. 
\end{proof}

 In the light of Theorem \ref{fg:Mal'cev} one may wonder whether Corollary~\ref{cor:class}
 generalizes to Mal'cev algebras.
This however is not the case, and it fails for loops in a very strong sense.
 Recall that the variety $\lL$ of loops has type 
$(\cdot,\backslash,/,1)$ (of arities $(2,2,2,0)$) and defining identities
\[
y\backslash (yx)=y(y\backslash x)=(xy)/y=(x/y)y=x1=1x=x;
\]
e.g. see \cite[Section II.1]{burris81}.

\begin{thm} \label{th:Ploop}
 In the variety of loops, no infinite direct product $\L\times \M$ with $\L$ and $\M$ non-trivial is
 finitely presented. 
\end{thm}

\begin{proof}
Suppose that $\L\times\M$ is finitely presented, and, without loss of generality, that $\L$ is infinite.
Both $\L$ and $\M$ must be finitely generated. Choose finite generating sets $X$ and $Y$ for $\L$ and $\M$ respectively such that neither contains the identity element.
Then the set $Z=(X\times\{1\})\cup (\{1\}\times Y)$
is a finite generating set for $\L\times\M$.

We are going to make use of Evans' solution to the word problem for finitely presented loops \cite{Ev:WPAA,Ev:MSD}.
As a consequence of \cite[Theorem 4.1]{Ev:MSD} and its proof there exists
a finite presentation $F_\lL(\overline{Z})/\Cg_{F_\lL(\overline{Z})}(R)$
for $\L\times\M$ with the following properties:
\begin{itemize}
\item
The set $\overline{Z}$ is finite, say $\overline{Z}=\{z_1,\ldots,z_n\}$ and contains $Z$.
(The assumption that the generators in $X$ and $Y$ are not equal to $1$ plays a role in ensuring that this is true.)
\item
All pairs in $R$ have the form $(z_i\circ z_j,z_k)$ where $\circ$ is one of the basic loop operations
$\cdot,\backslash,/$ and $z_i,z_j,z_k\in \overline{Z}$.
\item
Given a term $t(z_1,\dots,z_n)$ in variables $\overline{Z}$, the unique normal form of the element
$t^{\L\times\M}(z_1,\dots,z_m)$ can be obtained by successively, and in arbitrary order,
applying the following transformations: (1) replace $z_i\circ z_j$ by $z_k$ for any $(z_i\circ z_j,z_k)\in R$;
(2) replace $u$ by $v$ where $u=v$ is a defining identity for the variety $\lL$ and $|u|>|v|$.
\end{itemize}
(In Evans' terminology, $R$ is a \emph{closed} set of defining relations.)

Let $\overline{X}\subseteq L$ be such that
\[
\overline{X}\times \{1\}=\overline{Z}\cap (L\times\{1\}),
\]
and without loss suppose
\[
\overline{X}\times \{1\}=\{z_1,\ldots,z_m\}
\]
where $m\leq n$.
From $X\times \{1\}\subseteq Z\subseteq \overline{Z}$ it follows that $X\subseteq \overline{X}$,
and hence $\overline{X}$ is a generating set for $\L$.

\begin{claim}
For every $l\in L$ the normal form of $(l,1)\in L\times M$ has the form
$p(z_1,\dots,z_m)$ for some $m$-ary term $p$, i.e. it involves solely variables from $\overline{X}\times\{1\}$.
\end{claim}

\begin{proof}
Since $\overline{X}$ is a generating set for $\L\cong\L\times\{1\}$, there exists an $m$-ary term $q$ such
that $(l,1)=q^{\L\times\M}(z_1,\dots,z_m)$.
Consider the process of reducing the term $q(z_1,\dots,z_m)$ to its normal form according to the process described above.
If a rule $(z_i\circ z_j,z_k)$ is applied, then necessarily $z_i,z_j\in\overline{X}\times\{1\}$; but then
$z_k\in \overline{X}\times\{1\}$, as otherwise the equality would not hold in $\L\times\M$.
If a loop identity is applied, replacing the longer term by the shorter, no new variables are introduced.
It follows that after the first step we obtain another term over variables $z_1,\dots,z_m$, and an inductive argument completes the proof of the claim.
\end{proof}

Returning to the proof of the theorem, since $\L$ is infinite and $\overline{X}$ is finite, 
the set $(L\times{1})\setminus \{z_1,\dots,z_m\}$ is non-empty, and in particular there will exist a normal
form $z_i\circ z_j$ ($1\leq i,j\leq m$), which represents an element $(l,1)$ from this set.
Since $\M$ is non-trivial, the set $Y$ is non-empty; so let $y\in Y$ be arbitrary.
From $\{1\}\times Y\subseteq Z\subseteq\overline{Z}$ it follows that $(1,y)=z_k$ for some $k$ ($m<k\leq n$).
Now consider the terms $(z_i\circ z_j)\cdot z_k$ and $z_k\cdot (z_i\circ z_j)$.
Clearly, no rule from $R$ can be applied to either of those two terms, and neither can a loop identity.
Hence $(z_i\circ z_j)\cdot z_k$ and $z_k\cdot (z_i\circ z_j)$ are distinct normal forms with respect to the presentation $\langle \overline{X} | R\rangle$. On the other hand, these two terms clearly represent the same element $(l,y)$ of $\L\times \M$. This contradicts the fact that $\L\times\M\cong F_\lL(\overline{Z})/\Cg_{F_\lL(\overline{Z})}(R)$, and hence $\L\times\M$ is not finitely presented.
\end{proof}

Further classes in which this direction fails, even though the underlying result for finite generation holds, are
 provided by idempotent magmas (or groupoids, where we have a single
 binary operation symbol $\cdot$ and the defining identity $x\cdot x = x$) and lattices.

\begin{thm} \label{th:Pmagma}
 In the variety $\iI$ of idempotent magmas (groupoids),
 the direct product of two free algebras on two generators $F_\iI(x_1,x_2)^2$ is not finitely presented
 (even though it is finitely generated).
\end{thm}

\begin{proof}
 Consider the homomorphism
\[ h\colon F_\iI(x_{11},x_{12},x_{21},x_{22}) \to F_\iI(x_1,x_2)^2 \]
 that is defined by
\[ h(x_{ij}) = (x_i,x_j) \text{ for } i,j\in\{1,2\}. \] 
 We note that $h$ is onto by the idempotence of $\cdot$ (cf. the proof of
 Theorem~\ref{th:Gidempotent}). In particular $F_\iI(x_1,x_2)^2$ is finitely generated.

 To prove that $F_\iI(x_1,x_2)^2$ is not finitely presented, it suffices to show that $\ker h$ is not generated
 as a congruence by finitely many elements. So seeking a contradiction, we suppose that
 there exist $n\in\N$ and $(s_1,t_1),\dots,(s_n,t_n)\in\ker h$ such that
\[ \Cg_{F_\iI(x_1,x_2)^2}( (s_1,t_1),\dots,(s_n,t_n) ) = \ker h. \]
 Without loss of generality $s_i\neq t_i$ for all $i\leq n$.
 Note that every term over the binary operation $\cdot$ has odd length.
 Let $2k-1$ be the maximal length of the terms $s_1,\dots,s_n,t_1,\dots,t_n$. 

 Observe that for terms $s(x_{11},x_{12},x_{21},x_{22}),t(x_{11},x_{12},x_{21},x_{22})$ we have
\begin{equation} \label{eq:kerh}
 (s,t)\in\ker h \text{ iff } s(x_1,x_1,x_2,x_2) = t(x_1,x_1,x_2,x_2), s(x_1,x_2,x_1,x_2) = t(x_1,x_2,x_1,x_2). 
\end{equation}
 In the absence of parentheses we will associate all products on the left in what follows. Let
\[ u  = (x_{11}\underbrace{x_{21}\cdots x_{21}}_{k})(x_{12}\underbrace{x_{22}\cdots x_{22}}_{k}) \]
 and consider its congruence class modulo $\ker h$. 

 
 We first note that the defining identity of $\iI$ yields a confluent rewriting system where
 $tt$ can be replaced by $t$ for any term $t$. Hence we have normal forms for terms in the language
 of $\iI$ which are obtained by repeatedly performing such replacements in any order.
 Clearly $u$ is already in normal form.

 Let $w$ be a term that reduces to $u$ modulo the identities in $\iI$. Then $w = w_1w_2$
 for terms $w_1,w_2$ and one of the following holds:
\begin{enumerate}
\item either both $w_1$ and $w_2$ can be reduced to $u$ or
\item $w_1$ reduces to $x_{11}\underbrace{x_{21}\cdots x_{21}}_{k}$ and $w_2$ reduces to
 $x_{12}\underbrace{x_{22}\cdots x_{22}}_{k}$.
\end{enumerate}
 Consequently every subterm $t$ of $w$ of length at most $2k-1$ reduces to a subterm of $x_{11}x_{21}\cdots x_{21}$
 or of $x_{12}x_{22}\cdots x_{22}$. So $t$ is a term either in $\{x_{11},x_{21}\}$ or in $\{x_{12},x_{22}\}$.
 Hence $t$ is only congruent to itself modulo $\ker h$ by~\eqref{eq:kerh}. In particular none of the relations
 $(s_1,t_1),\dots,(s_n,t_n)$ can be applied to $t$ or to $w$. It follows that $u$ is congruent only to
 itself modulo $\ker h$. This contradicts the fact that by~\eqref{eq:kerh} and idempotence
\[  u \equiv (x_{11}x_{12})\underbrace{(x_{21}x_{22})\cdots(x_{21}x_{22})}_k \mod \ker h. \] 
 %
 Thus $\ker h$ is not finitely generated, and $F_\iI(x_1,x_2)^2$
 is not finitely presented. 
\end{proof}

\begin{thm} \label{th:Lfp}
 In the variety of lattices, no direct product of two infinite lattices is finitely presented.
\end{thm}


\begin{proof}
 Seeking a contradiction let $\A,\B$ be infinite lattices such that $\A\times\B$ is finitely presented. 
 By~\cite[Theorem 1]{GHL:SFPL} (see also~\cite[Theorem 95]{Gr:LTF}) every finitely presented lattice
 $\L$ has a congruence $\rho$ such that $\L/\rho$ is finite and every congruence class embeds into some
 free lattice. Note that each congruence class is a sublattice because of the idempotence of the operations.
 Since all congruences on $\A\times\B$
 are product congruences (by \cite[Theorem 149]{Gr:LTF} and \cite[Lemma IV.11.10]{burris81}),
 this implies we have $\alpha\in\Con(\A),\beta\in\Con(\B)$ such that $\A/\alpha$ and
 $\B/\beta$ are finite and every $\alpha\times\beta$-class in $\A\times\B$ embeds into a free lattice.
 Let $a\in A, b\in B$ such that $a/\alpha$ and $b/\beta$ are infinite.
 By a result of J\'onsson~\cite[Theorem 2.7]{Jo:SFL} (see also~\cite[Theorem 5.4.5]{Gr:LTF})
 every sublattice of finite length (i.e., without infinite chain) in a free lattice is finite.
 Since $a/\alpha$ and $b/\beta$ are infinite and embed into a free lattice, each must contain either a copy of the
 infinite chain $\Nl = \algop{\N}{\max,\min}$ or its dual $\Nl^\prime=\algop{\N}{\min,\max}$.
In particular, both $a/\alpha$ and $b/\beta$ contain a copy of the 4-element chain
$\mathbf{C}_4=\algop{\{0,1,2,3\}}{\max,\min}$, and
so $\mathbf{C}_4\times\mathbf{C}_4$ embeds into $(a,b)/(\alpha\times\beta)$ which in turn embeds
 into some free lattice.

 Now every free lattice $\F$ satisfies Whitman's condition~\cite[Theorem 5.44]{Gr:LTF},
namely that for all $x,y,u,v\in F$ we have
\[ x\wedge y \leq u\vee v \Longleftrightarrow x \leq u\vee v \text{ or } y \leq u\vee v \text{ or }  x\wedge y \leq u \text{ or }  x\wedge y \leq v. \]
 But Whitman's condition does not hold in $\mathbf{C}_4\times\mathbf{C}_4$ as witnessed by $x=(1,3),y=(3,1),u=(0,2),v=(2,0)$.
 So $(a,b)/(\alpha\times\beta)$ does not embed in any free lattice contradicting our assumption.
 Thus $\A\times\B$ is not finitely presented.
\end{proof}


Finally we recall that the situation for semigroups is quite subtle, and is dealt in some detail in \cite{Robertson98,Araujo00}.

\section{Residual finiteness}
\label{secrf}

 An algebra $\A$ is \emph{residually finite} if for all distinct $a,b\in A$ there exists a finite algebra $\C$
 and a homomorphism $f$ from $\A$ to $\C$ such that $f(a) \neq f(b)$. Equivalently, there exists a congruence
 $\rho$ of $\A$ such that $A/\rho$ is finite and $a\not\equiv_\rho b$.
 Recall that it was an easy and completely general fact that if $\A\times \B$ is finitely generated, then so are
 $\A$ and $\B$.
In the case of residual finiteness it is precisely the other direction that always holds:

\begin{prop}
If $\A$ and $\B$ are residually finite algebras, then so is $\A\times\B$.
\end{prop}

\begin{proof}
For distinct $(a,b),(c,d)\in A\times B$ we must have $a\neq c$ or $b\neq d$. Without loss assume $a\neq c$.
Since $\A$ is residually finite, there exists a homomorphism $f:\A\rightarrow \C$ with $\C$ finite and $f(a)\neq f(c)$. The composition of $f$ with the projection $\pi_\A : \A\times\B \rightarrow \A$ yields a homomorphism $\A\times\B\rightarrow \C$ that separates $(a,b)$ and $(c,d)$.
\end{proof}

In fact, the converse is often just as obvious. For instance, it is clear that residual finiteness is a hereditary
 property, i.e. if $\A$ is residually finite and $\B$ is a subalgebra of $\A$, then $\B$ is residually finite as well.
 Thus the if and only if statement holds for all classical algebraic structures, monoids, loops, idempotent algebras,
 in particular lattices, and many other classes, because factors embed into the direct product.
 In \cite{Gray09} the authors prove that the analogous statement holds for semigroups,
 albeit for not altogether trivial reasons, and provide examples which show that it does sometimes fail (in particular for unary algebras).
 Here we prove the result for congruence modular varieties.

\begin{thm}  \label{th:RCM}
 Let $\A$ and $\B$ be algebras in a congruence modular variety.
 Then $\A\times\B$ is residually finite if and only if both $\A$ and $\B$ are residually finite. 
\end{thm}

\begin{proof}
 Our argument requires some commutator theory for congruence modular varieties as developed in~\cite{FM:CTC}.
 Assume that $\A\times\B$ is residually finite. To show that $\A$ is residually finite, we consider distinct
 elements $a_1,a_2\in A$ and fix some $b_1\in B$. Since $\A\times\B$ is residually finite,
 we have $\rho_1\in\Con(\A\times\B)$ of finite index such that $(a_1,b_1) \not\equiv_{\rho_1} (a_2,b_1)$.

 Note that the product congruences of $\A\times\B$ form a sublattice of $\Con(\A\times\B)$. 
 In particular, there exists a minimal product congruence $\alpha\times\beta$ above $\rho_1$ 
 and a maximal product congruence $\gamma\times\delta$ below $\rho_1$. 
Note that $\alpha\times\beta$ must have finite index, because it contains $\rho_1$ which has finite index; hence both $\alpha$ and $\beta$ are congruences of finite index.
If $\rho_1$ itself is a product of
 congruences, then $\alpha=\gamma$ separates  $a_1,a_2$ and the result is proved.
 
 So we assume that $\rho_1 < \alpha\times\beta$ is a skew congruence. By factoring $\A\times\B$ by
 $\gamma\times\delta$ we may assume that $0_A\times 0_B$ is the only product congruence below $\rho_1$.
 We first claim that
\begin{equation} \label{eq:abc}
 \beta \text{ is abelian, i.e. } [\beta,\beta] = 0_B.
\end{equation}
 From the basic properties of the commutator~\cite[Propositon 3.4]{FM:CTC} we have that
 $[\rho_1,0_A\times\beta] \leq\rho_1$ and
 $[1_A\times 0_B,0_A\times\beta] = [1_A,0_A]\times [0_B,\beta] = 0_A\times 0_B \leq\rho_1$.
 Since the commutator is join distributive by~\cite[Exercise 2.2]{FM:CTC}, this yields
\begin{equation} \label{eq:rv01}
 [\rho_1\vee (1_A\times 0_B),0_A\times \beta] \leq\rho_1. 
\end{equation}
 By Lemma~\ref{le:rs} we have that $\rho_1\vee (1_A\times 0_B)$ is a product congruence that
 contains $\rho_1$ and hence contains $\alpha\times\beta$ by minimality. Since the commutator
 is monotone~\cite[Proposition 3.4]{FM:CTC},~\eqref{eq:rv01} yields 
\[ [\alpha\times\beta,0_A\times\beta] \leq\rho_1. \]
 So $[\alpha\times\beta,0_A\times\beta] = 0_A\times [\beta,\beta]$ is a product congruence
 that is contained in $\rho_1$ and hence it is $0_A\times 0_B$. Thus~\eqref{eq:abc} is proved.

Recall that $\beta$ has finite index, say $|B/\beta| =: n$, and extend $b_1$ to a complete set of representatives $b_1,\dots,b_n$ 
 for the $\beta$-classes in $B$. For every $i\in\{2,\dots,n\}$ we have $\rho_i\in\Con(\A\times\B)$
 of finite index such that $(a_1,b_i) \not\equiv_{\rho_i} (a_2,b_i)$.
 Then $\rho = \bigcap_{i=1}^n \rho_i$ has finite index in $\A\times\B$, $\rho < \alpha\times\beta$ and
\begin{equation} \label{eq:abi}
 (a_1,b_i) \not\equiv_{\rho} (a_2,b_i) \text{ for all } i\in\{1,\dots,n\}. 
\end{equation}
 By Lemma~\ref{le:rs}~\eqref{it:rs} we have a congruence
\[ \sigma = \{ (u,v)\in A^2 \setsuchthat (u,b)\equiv_\rho (v,b) \text{ for some } b\in B \} \] 
 on $\A$. 

 Let $m = |(A\times B)/\rho|$. We claim that $|A/\sigma| \leq m$. Suppose otherwise that
 $u_1,\dots,u_{m+1}\in A$ are pairwise distinct modulo $\sigma$. 
 Then for any $b\in B$ we have that $(u_1,b)\dots,(u_{m+1},b)$ are pairwise distinct modulo
 $\rho$ which contradicts our assumption that $\rho$ has only $m$ classes.
 Hence $|A/\sigma|\leq |A\times B/\rho|$.

 Seeking a contradiction we now suppose that $a_1\equiv_\sigma a_2$. Then we have
 $b\in B$ and some $i\in\{1,\dots,n\}$ such that $b\equiv_\beta b_i$ and 
\begin{align*} 
 (a_1,b_i) & = (a_1,b_i) \\ 
 (a_1,b) & = (a_1,b) \\
 (a_1,b) & \equiv_\rho (a_2,b)
\end{align*}
 By~\cite[Theorem 5.5]{FM:CTC} each congruence modular variety has a \emph{difference term},
 which is a ternary term $d$ which satisfies the following for all algebras $\C$ in the variety:
 \[ d^\C(x,x,y) = y \text{ and } d^\C(x,y,y) \equiv x \bmod [\Cg_\C(x,y),\Cg_\C(x,y)] \text{ for all } x,y\in C. \]
 Together with~\eqref{eq:abc}, $b\equiv_\beta b_i$ implies $d^\B(b_i,b,b) = b_i$. So applying $d$ to the
 three equations above yields
\[ (a_1,b_i) \equiv_\rho (a_2,b_i) \]
 which contradicts~\eqref{eq:abi}. Hence $a_1\not\equiv_\sigma a_2$ and $\A$ is residually finite.
\end{proof}

 As an immediate consequence, we obtain the if and only if statement for all the `classical' algebraic
 structures:

\begin{cor}
 Let $\vV$ be any of the following varieties: groups, rings, modules over a ring $R$,
 associative algebras or Lie algebras over a field, lattices, quasigroups, or loops, and let $\A,\B\in\vV$.
 Then $\A\times\B$ is residually finite if and only if $\A$ and $\B$ are residually finite.
\end{cor}

\section{Conclusion}

 Throughout this paper we have seen various situations where finiteness properties are not preserved by 
 direct products and settings where they are by somehow non-trivial reasons.
 We close the paper by giving some questions and suggestions for future research that -- to our mind -- would further enhance our
 understanding of these phenomena.

 In the variety of semigroups the direct product of free semigroups is not finitely presented in general.
 Still certain direct products of infinite, finitely presented semigroups are finitely presented.
This is quite different from our result for loops (Theorem~\ref{th:Ploop}) where no direct product is ever finitely presented except for trivial reasons.
 
\begin{que}
 Is there a congruence modular variety where direct products of free algebras are not finitely presented
 but some product of non-trivial algebras is infinite and finitely presented?
\end{que}

 For loops, and also idempotent magmas (Theorem~\ref{th:Pmagma}),
 it is the non-associativity of operations that seems to be an obstacle for the preservation of finite
 presentability in direct products.

\begin{que}
 Is finite presentability preserved in some varieties of loops that satisfy some generalized associativity
 law, like Bol loops or Moufang loops.
\end{que}

 Theorem~\ref{th:Lfp} shows that it is difficult for an infinite direct product of lattices to be finitely presented
 without ruling it out altogether.

\begin{que}
 Is there an infinite, finitely presented lattice $\L$ and a non-trivial lattice $\M$ (necessarily finite)
 such that $\L\times\M$ is finitely presented?
\end{que}

 Throughout our paper unary algebras provide examples of unexpected behaviour.

\begin{que}
 Classify classes of unary algebras where direct products preserve finite generation, finite presentability,
 or residual finiteness.
\end{que}

 Congruence modular varieties provided a broad context in which direct products were well-behaved with respect
 to residual finiteness and finite presentability (from products to factors).

\begin{que}
 Do our results on finite presentability and residual finiteness also hold for generalizations of 
 congruence modular varieties, e.g. varieties with Taylor term~\cite{HM:SFA}?
\end{que}


 Another possible direction for future research is to consider preservation of finiteness conditions under
 subdirect products, which play a prominent role in the theory of groups~\cite{BM:SFP,Bridson13} and
 more generally Mal'cev algebras.


\section*{Acknowledgements}

The authors are grateful to Erhard Aichinger, Ralph Freese, Keith Kearnes, Michael Kinyon, \'{A}gnes Szendrei and Petr Vojt\v{e}chovsk\'{y} for helpful conversations on the material in this paper.

\end{document}